\documentclass[12pt,final]{article}
\usepackage[margin=2.5cm,top=1.5cm]{geometry}
\usepackage[all,arc]{xy}
\usepackage{mathrsfs}
\usepackage{tikz}
\usepackage{graphicx}
\usepackage{mathtools}
 \usepackage{subfigure}
\usepackage[centerlast]{caption2}
\usepackage{amsmath,amsthm,amssymb,enumerate}
\usepackage{color}
\usepackage{setspace}

\newtheorem{defn}{Definition}[section]
\newtheorem{lem}[defn]{Lemma}
\newtheorem{theo}[defn]{Theorem}
\newtheorem{cor}[defn]{Corollary}

\newtheorem{prop}[defn]{Proposition}

\newtheorem{prob}[defn]{Question}

\newtheorem{rem}[defn]{Remark}

\numberwithin{equation}{section}

\newcommand\keywordsname{Key words}
\newcommand\AMSname{AMS subject classifications}

\newenvironment{@abssec}[1]{%
     \if@twocolumn
       \section*{#1}%
     \else
       \vspace{.05in}\footnotesize
       \parindent .2in
         {\upshape\bfseries #1. }\ignorespaces
     \fi}
     {\if@twocolumn\else\par\vspace{.1in}\fi}

\begin{document}

\title{Sharp upper and lower bounds for the spectral radius of a nonnegative irreducible matrix and its
applications\footnote{L. You's research is supported by the Zhujiang Technology New Star Foundation of Guangzhou (Grant No. 2011J2200090)
and Program on International Cooperation and Innovation,
Department of Education, Guangdong Province (Grant No. 2012gjhz0007),
P. Yuan's research is supported by the NSF of China (Grant No. 11271142) and
the Guangdong Provincial Natural Science Foundation(Grant No. S2012010009942).}}

\author{ Lihua You\footnote{{\it{Email address:\;}}ylhua@scnu.edu.cn.} 
 \qquad Yujie Shu\footnote{{\it{Email address:\;}}1020697000@qq.com. } 
 \qquad Pingzhi Yuan\footnote{{\it{Corresponding author:\;}}yuanpz@scnu.edu.cn. }
 }\vskip.2cm
\date{{\small
School of Mathematical Sciences, South China Normal University,\\
Guangzhou, 510631, P.R. China\\
}}
\maketitle

\noindent {\bf Abstract }
In this paper, we obtain 
 the sharp upper and lower bounds for the spectral radius of a nonnegative irreducible matrix.
 We also apply these bounds to various matrices associated with a graph or a digraph,
  obtain some new results or known results about various spectral radii,
  including  the adjacency spectral radius,
  the signless Laplacian spectral radius, the distance spectral radius, the distance signless Laplacian spectral radius of a graph or a digraph.



\section{Introduction}

\hskip.6cm
We begin by recalling some definitions.
Let $M$ be an $n\times n$ real matrix, $\lambda_1, \lambda_2, \ldots, \lambda_n$ be the eigenvalues of $M$.
 It is obvious that the eigenvalues may be complex numbers since $M$ is not symmetric in general.
 We usually assume that $|\lambda_1|\geq |\lambda_2|\geq\ldots \geq |\lambda_n|$.
 The spectral radius of $M$ is defined as $\rho(M)=|\lambda_1|$, i.e., it is the largest modulus
 of the eigenvalues of $M$.   If $M$ is a nonnegative matrix, it follows from the Perron-Frobenius theorem that
 the spectral radius $\rho(M)$ is a eigenvalue of $M$.
 If $M$ is a nonnegative irreducible  matrix, it follows from the Perron-Frobenius theorem that
 $\rho(M)=\lambda_1$ is simple.

Let $G=(V,E)$ be a simple graph with vertex set $V=V(G)=\{v_1,v_2,\ldots,v_n\}$ and edge set $E=E(G)$.
  Let $A(G)=(a_{ij})$ be the $(0,1)$ adjacency matrix of $G$ where $a_{ij}=1$ if $v_i$ and $v_j$ are adjacent and 0 otherwise.
  Let  $d_i$ be the degree of vertex $v_i$, $diag(G)=diag(d_1, d_2, \ldots, d_n)$ be the diagonal matrix of vertex degrees of $G$.
   Then the  signless Laplacian matrix of $G$ is defined as
 $$ Q(G)=diag(G)+A(G).$$
  The spectral radius of $A(G)$  and $Q(G)$,  denoted by $\rho(G)$ and $q(G)$,
 are called the (adjacency) spectral radius of $G$ and  the  signless Laplacian spectral radius of $G$, respectively.

Let $G=(V,E)$ be a connected graph with vertex set $V=V(G)=\{v_1,v_2,\ldots,v_n\}$ and edge set $E=E(G)$.
For $u,v\in V$, the distance between $u$ and $v$, denoted by $d_G(u,v)$,
is the length of the shortest path connecting them in $G$.
For $u\in V$, the transmission of vertex $u$ in $G$ is the sum of distances  between $u$ and
all other vertices of $G$, denoted by $Tr_G(u)$.

The distance matrix of $G$ is the $n\times n$ matrix $\mathcal{D}(G)=(d_{ij})$ where $d_{ij}=d_G(v_i,v_j)$.
In fact, for $1\leq i\leq n$, the transmission of vertex $v_i$, $Tr_G(v_i)$ is just the $i$-th row sum
 of $\mathcal{D}(G)$.  So for convenience,  we also call $Tr_G(v_i)$  the distance degree of vertex $v_i$ in $G$,
 denoted by  $D_i$, that is, $D_i=\sum\limits_{j=1}^nd_{ij}=Tr_G(v_i)$.

 Let $Tr(G)=diag(D_1, D_2, \ldots, D_n)$ be the diagonal matrix of vertex transmissions of $G$.
  The distance signless Laplacian matrix of $G$ is
   the $n\times n$ matrix defined by Aouchiche and Hansen as (\cite{2013A})
 $$ \mathcal{Q}(G)=Tr(G)+\mathcal{D}(G).$$

 \noindent The spectral radius of $\mathcal{D}(G)$  and $\mathcal{Q}(G)$,  denoted by $\rho^{\mathcal{D}}(G)$ and  $q^{\mathcal{D}}(G)$, are called the distance spectral radius of $G$  and  the distance signless Laplacian spectral radius of $G$, respectively.

Let $\overrightarrow{G}=(V, E)$ be a digraph,
where $V=V(\overrightarrow{G})=\{v_1, v_2,
\ldots, v_n \}$ and $E=E(\overrightarrow{G})$ are the  vertex set and arc set of $\overrightarrow{G}$,
respectively. A digraph $\overrightarrow{G}$ is simple if it has no loops and
multiple arcs. A digraph  $\overrightarrow{G}$ is strongly connected if for every
pair of vertices $v_i, v_j\in V$, there are  directed paths from $v_i$
to $v_j$ and from $v_j$ to $v_i$. In this paper, we consider finite, simple digraphs.

Let $\overrightarrow{G}$ be a digraph.
Let $N^+_{\overrightarrow{G}}(v_i)=\{v_j\in V(\overrightarrow{G})|$ $(v_i, v_j) \in E(\overrightarrow{G})\}$
denote the set of out-neighbors of $v_i$,
$d^+_i=|N^+_{\overrightarrow{G}}(v_i)|$ denote the out-degree of the vertex $v_i$ in $\overrightarrow{G}$.

 For a digraph $\overrightarrow{G}$, let $A(\overrightarrow{G})=(a_{ij})$ denote the adjacency matrix of $\overrightarrow{G}$,
 where $a_{ij}$ is equal to the number of arcs $(v_i, v_j)$.
 Let $diag(\overrightarrow{G})=diag(d^+_1, d^+_2, \ldots, d^+_n)$ be the diagonal matrix of the vertex out-degrees of  $\overrightarrow{G}$ and $$
 Q(\overrightarrow{G})= diag(\overrightarrow{G})+A(\overrightarrow{G})$$
 be the signless Laplacian matrix of $\overrightarrow{G}$.
 The spectral radius of $A(\overrightarrow{G})$ and $Q(\overrightarrow{G})$,
  denoted by $\rho(\overrightarrow{G})$ and $q(\overrightarrow{G})$,
  are called the (adjacency) spectral radius of $\overrightarrow{G}$
  and the signless Laplacian spectral radius of $\overrightarrow{G}$, respectively.

For $u,v\in V(\overrightarrow{G})$, the distance from $u$ to $v$, denoted by $d_{\overrightarrow{G}}(u,v)$,
is the length of the shortest directed path from $u$ to $v$ in ${\overrightarrow{G}}$.
For $u\in V({\overrightarrow{G}})$, the transmission of vertex $u$ in ${\overrightarrow{G}}$ is the sum of distances
 from $u$ to all other vertices of ${\overrightarrow{G}}$, denoted by $Tr_{{\overrightarrow{G}}}(u)$.

Let ${\overrightarrow{G}}$ be a strong connected digraph with vertex set $V({\overrightarrow{G}})=\{v_1, v_2, \ldots, v_n\}$.
The distance matrix of ${\overrightarrow{G}}$ is the $n\times n$ matrix $\mathcal{D}({\overrightarrow{G}})=(d_{ij})$
where $d_{ij}=d_{\overrightarrow{G}}(v_i,v_j)$.
In fact, for $1\leq i\leq n$, the transmission of vertex $v_i$, $Tr_{\overrightarrow{G}}(v_i)$ is just the $i$-th row sum
 of $\mathcal{D}(\overrightarrow{G})$.  So for convenience,
 we also call $Tr_{\overrightarrow{G}}(v_i)$  the distance degree of vertex $v_i$ in $\overrightarrow{G}$,
 denoted by  $D_i^+$, that is, $D_i^+=\sum\limits_{j=1}^nd_{ij}=Tr_{\overrightarrow{G}}(v_i)$.

 Let  $Tr(\overrightarrow{G})=diag(D_1^+, D_2^+, \ldots, D_n^+)$
 be the diagonal matrix of vertex transmissions of $\overrightarrow{G}$.
 The distance  signless Laplacian matrix of ${\overrightarrow{G}}$ is the $n\times n$ matrix defined similar to the undirected graph by Aouchiche and Hansen as (\cite{2013A})
 $$ \mathcal{Q}({\overrightarrow{G}})=Tr({\overrightarrow{G}})+\mathcal{D}({\overrightarrow{G}}).$$
 \noindent The spectral radius of $\mathcal{D}({\overrightarrow{G}})$ and
 $\mathcal{Q}({\overrightarrow{G}})$, denoted by $\rho^{\mathcal{D}}(\overrightarrow{G})$ and $q^{\mathcal{D}}(\overrightarrow{G})$,
 are called the distance  spectral radius of $\overrightarrow{G}$
 and the distance signless Laplacian spectral radius of $\overrightarrow{G}$, respectively.

Let $G=(V,E)$ be a graph, for $v_i, v_j\in V$, if $v_i$ is adjacent to $v_j$,  we denote it by $i\sim j$.
Moreover,  we call $m_i=\frac{\sum\limits_{i\sim j}{d_j}}{d_i}$ the average degree of the neighbors of $v_i$.
If $G$ is connected, we call $T_i=\sum\limits_{j=1}^n{d_{ij}D_j}$ the second distance degree of $v_i$ in $G$,
where $D_i=\sum\limits_{j=1}^nd_{ij}=Tr_G(v_i)$ is the distance degree of vertex $v_i$ in $G$.

Let $\overrightarrow{G}=(V,E)$ be a digraph, for $v_i, v_j\in V$, if arc $(v_i, v_j)\in E$, we denoted it by $i\sim j$.
Moreover, we call $m^{+}_i=\frac{\sum\limits_{i\sim j}{d^{+}_j}}{d^{+}_i}$ the average out-degree of the out-neighbors of $v_i$,
where $d^+_i$ is the out-degree of vertex $v_i$ in $\overrightarrow{G}$.
If $\overrightarrow{G}$ is strong connected, we call $T^+_i=\sum\limits_{j=1}^n{d_{ij}D^+_j}$ the second distance out-degree of $v_i$ in $\overrightarrow{G}$,
where $D^+_i=\sum\limits_{j=1}^{n}{d_{ij}}=Tr_{\overrightarrow{G}}(v_i)$ is the distance out-degree of vertex $v_i$ in $\overrightarrow{G}$.

A regular graph is a graph where every vertex has the same degree.
A bipartite semi-regular  graph is a bipartite graph $G=(U,V,E)$ for which every two vertices on the same side of the given bipartition have the same degree as each other.

So far, there are many results on the bounds of the  spectral radius of a nonnegative matrix, the  spectral radius, the signless Laplacian spectral radius, the distance spectral radius and the distance signless Laplacian spectral radius of a graph and a  digraph, see [1,3-5,7,8,10-18]. The following  are some results on the above spectral radius of a graph or a digraph in terms of
degree, average degree, distance degree, the second distance degree or
out-degree, average out-degree, distance out-degree, the second distance out-degree and so on.
\vskip.1cm
$\rho(G)\leq \max \limits_{1\leq i\leq n}\big\{d_im_i\big\}  $ \hskip11.5cm (1.1)  

\vskip.1cm
$ \rho(G)\leq \max\limits_{1\leq i,j\leq n}\big\{\sqrt{m_im_j}, i\sim j \big\} $ \hskip9.7cm (1.2)

\vskip.1cm
$ q(G)\leq \max\limits_{1\leq i\leq n}\big\{d_i+\sqrt{d_im_i}\big\} $ \hskip10.55cm (1.3) 

\vskip.1cm
$ q(G)\leq \max\limits_{1\leq i\leq n}\bigg\{\frac{d_i+\sqrt{d^2_i+8d_im_i}}{2}\bigg\} $ \hskip9.97cm (1.4) 

\vskip.1cm
$ \rho^D(G)\leq \max\limits_{1\leq i,j\leq n}\bigg\{\sqrt{\frac{T_iT_j}{D_iD_j}}\bigg\} $ \hskip10.47cm (1.5) 

\vskip.1cm
$ \min\limits_{1\leq i\leq n}\bigg\{\frac{T_i}{D_i}\bigg\}\leq \rho^D(G)\leq \max\limits_{1\leq i\leq n}\bigg\{\frac{T_i}{D_i}\bigg\} $ \hskip8.9cm (1.6) 

\vskip.1cm
$ \min\limits_{1\leq i\leq n}\bigg\{\sqrt{T_i}\bigg\}\leq \rho^D(G)\leq \max\limits_{1\leq i\leq n}\bigg\{\sqrt{T_i}\bigg\} $ \hskip8.35cm (1.7) 

\vskip.1cm
$ q^D(G)\leq \max\limits_{1\leq i,j\leq n}\bigg\{\frac{D_i+D_i+\sqrt{(D_i-D_j)^2+\frac{4T_iT_j}{D_iD_j}}}{2}\bigg\}$ \hskip7.82cm (1.8) 

\vskip.2cm
$ \min\limits_{1\leq i\leq n}\big\{D_i+\frac{T_i}{D_i}\bigg\}\leq q^D(G)\leq \max\limits_{1\leq i\leq n}\big\{D_i+\frac{T_i}{D_i}\bigg\}$ \hskip7.49cm (1.9) 

\vskip.2cm
$ \min\limits_{1\leq i\leq n}\big\{\sqrt{2T_i+2D^2_i}\big\}\leq q^D(G)\leq \max\limits_{1\leq i\leq n}\big\{\sqrt{2T_i+2D^2_i}\big\}$ \hskip5.62cm (1.10) 

\vskip.2cm
$ \min\{d^+_i:v_i\in V(\overrightarrow{G})\}\leq \rho(\overrightarrow{G})\leq \max\{d^+_i:v_i\in V(\overrightarrow{G})\}$ \hskip5.42cm (1.11) 

\vskip.2cm
$ \min\{m^+_i:v_i\in V(\overrightarrow{G})\}\leq \rho(\overrightarrow{G})\leq \max\{m^+_i:v_i\in V(\overrightarrow{G})\}$ \hskip5.14cm (1.12) 

\vskip.2cm
$ \min\{\sqrt{d^+_im^+_i}:v_i\in V(\overrightarrow{G})\}\leq \rho(\overrightarrow{G})\leq \max\{\sqrt{d^+_im^+_i}:v_i\in V(\overrightarrow{G})\}$ \hskip3.41cm (1.13) 

\vskip.2cm
$ \min\{\sqrt{\frac{\sum\limits_{i\sim j}{d^+_jm^+_j}}{d^+_i}}:v_i\in V(\overrightarrow{G})\}\leq \rho(\overrightarrow{G})\leq \max\{\sqrt{\frac{\sum\limits_{i\sim j}{d^+_jm^+_j}}{d^+_i}}:v_i\in V(\overrightarrow{G})\}$ \hskip2.75cm (1.14) 

\vskip.2cm
$ \min\{\sqrt{m^+_im^+_j}:i\sim j\}\leq \rho(\overrightarrow{G})\leq \max\{\sqrt{m^+_im^+_j}:i\sim j\}$ \hskip5.15cm (1.15) 

\vskip.2cm
$ \min\{d^+_i+m^+_i:v_i\in V(G)\}\leq q(\overrightarrow{G})\leq \max\{d^+_i+m^+_i:v_i\in V(G)\}$ \hskip3.79cm (1.16) 

\vskip.2cm
$ \min\limits_{i\sim j}\bigg\{\frac{d^+_i+d^+_j+\sqrt{(d^+_i-d^+_j)^2+4m^+_im^+_j}}{2}\bigg\}\leq q(\overrightarrow{G})\leq\max\limits_{i\sim j}\bigg\{\frac{d^+_i+d^+_j+\sqrt{(d^+_i-d^+_j)^2+4m^+_im^+_j}}{2}\bigg\}$ \hskip1.79cm (1.17) 

\vskip.2cm
$ q(\overrightarrow{G})\leq \max\limits_{1\leq i\leq n} \bigg\{d^+_i+\sqrt{\sum\limits_{j\sim i}{d^+_j}}\bigg\}$ \hskip9.96cm (1.18) 

\vskip.2cm
$ \min\limits_{1\leq i\leq n}{D_i}\leq\rho^D(\overrightarrow{G})\leq \max\limits_{1\leq i\leq n}{D_i}$ \hskip10.06cm (1.19) 

\vskip.2cm
$ \min\limits_{1\leq i,j\leq n}{\sqrt{D_iD_j}}\leq\rho^D(\overrightarrow{G})\leq \max\limits_{1\leq i\leq n}{\sqrt{D_iD_j}}$ \hskip8.0cm (1.20) 

\vskip.2cm

In this paper, we obtain the
 sharp upper and lower bounds for the spectral radius of a nonnegative irreducible matrix in Section 2,
and then we apply these bounds to various matrices associated with a graph in Section 3 or  a digraph in Section 4,
  obtain some new results or known results about various spectral radii,
  including  the (adjacency)  spectral radius,
  the signless Laplacian spectral radius, the distance spectral radius, the distance signless Laplacian spectral radius  of a graph or a digraph.

\section{Main result}
\hskip.6cm In this section, we will obtain the sharp upper and lower bounds for the spectral radius of a nonnegative irreducible matrix.
Applying the result, we will point out the necessity and sufficiency conditions of the equality holding
in Theorem 2.4 in \cite{2014} are incorrect. The techniques used in this section is motivated by \cite{2014} et al.


\begin{lem}(\cite{2005}) \label{lem23}
Let $A$ be a nonnegative matrix with the spectral radius $\rho(A)$ and the row sum $r_1,r_2,\ldots,r_n$. Then
$\min\limits_{1\leq i\leq n}r_i\leq\rho(A)\leq \max\limits_{1\leq i\leq n}r_i.$
Moreover, if $A$ is an irreducible matrix, then one of equalities holds if and only if the row sums of $A$ are all equal.
\end{lem}

\begin{theo}\label{T2}
Let $A=(a_{ij})$ be an $n\times n$  nonnegative  irreducible matrix with $a_{ii}=0$ for  $i=1,2,\ldots,n$, and the row sum $r_1,r_2,\ldots,r_n$.
Let $B=A+M$, where $M=diag(t_1,t_2,\ldots,t_n)$ with $t_i\geq0$ for any $i\in\{1,2,\ldots,n\}$, $s_i=\sum\limits_{j=1}^n{a_{ij}r_j}$,
$\rho(B)$ be the spectral radius of $B$.
Let $f(i,j)=\frac{t_i+t_j+\sqrt{(t_i-t_j)^2+\frac{4s_is_j}{r_ir_j}}}{2}$ for any $1\leq i,j\leq n$.
Then
\begin{equation}\label{eq25}
\min\limits_{1\leq i,j\leq n}\{f(i,j), a_{ij}\not=0\}\leq \rho(B)
\leq \max\limits_{1\leq i,j\leq n}\{f(i,j), a_{ij}\not=0\}.
\end{equation}
Moreover, one of the equalities  in (\ref{eq25}) holds if and only if one of the two conditions holds:

{\rm (i) } $t_i+\frac{s_i}{r_i}=t_j+\frac{s_j}{r_j}$ for any $i\in\{1,2,\ldots,n\}$;

{\rm (ii) } There exists an integer $k$ with $1\leq k<n$  such that $B$ is a partitioned matrix, where
 \begin{equation}\label{eq26}
 B=
\left(\begin{array}{cccc|cccc}
t_1     & 0        & \ldots & 0       & a_{1,k+1} & a_{1,{k+2}} & \ldots & a_{1n} \\
0       & t_2      & \ldots & 0       & a_{2,k+1} & a_{2,{k+2}} & \ldots & a_{2n} \\
\vdots  & \vdots   & \ddots & \vdots  &\vdots    & \vdots       & \ddots & \vdots \\
0       & 0        & \ldots & t_{k} & a_{k,{k+1}} & a_{k,{k+2}} & \ldots & a_{kn} \\
\hline
 a_{k+1,1} & a_{k+1,2} & \ldots & a_{k+1,k}     & t_{k+1}     & 0              & \ldots     & 0 \\
 a_{k+2,1} & a_{k+2,2} & \ldots & a_{k+2,k}     &0               & t_{k+2}     & \ldots & 0 \\
\vdots  & \vdots   & \ddots & \vdots           &\vdots         & \vdots        & \ddots & \vdots \\
 a_{n1} & a_{n2}  & \ldots & a_{nk}          &0              & 0             & \ldots & t_{n} \\
\end{array}\right),
\end{equation}
and there exists $l>0$ such that $t_1+\frac{ls_1}{r_1}=\ldots=t_{k}+\frac{ls_{k}}{r_{k}}=t_{k+1}+\frac{s_{k+1}}{lr_{k+1}}=\ldots=t_{n}+\frac{s_n}{lr_n}$.
In fact,  $l>1$ when the left equality holds and $l<1$ when the right equality holds.
\end{theo}

\begin{proof}
Let $R=diag(r_1,r_2,\ldots,r_n)$.
Since $A$ is a nonnegative irreducible matrix, then $B=(b_{ij})$, $R^{-1}BR$ are nonnegative  irreducible,
and $B$, $R^{-1}BR$ have the same eigenvalues, where $b_{ij}=\left\{\begin{array}{cc}
                                                                      t_i, & \mbox{ if } i=j; \\
                                                                      a_{ij},  & \mbox{ if } i\not=j.
                                                                    \end{array}\right.$
By the Perron-Frobenius theorem, we can assume that $X=(x_1,x_2,\ldots,x_n)^T$ be a positive eigenvector of  $R^{-1}BR$
corresponding to the eigenvalue $\rho(B)$.

Upper bounds: Without loss of generality, we can assume that one entry of $X$, say $x_p$,
is equal to 1 and the others are less than or equal to 1, i.e. $x_p=1$ and $0<x_k\leq1$ for all others $1\leq k\leq n$.
Let $x_q=\max\{x_k \mid a_{pk}\neq 0, 1\leq k\leq n\}$, it is clear that  $q\not= p$, $a_{pq}\not=0$ and $x_q\leq x_p$.
 By $R^{-1}BRX=\rho(B)X,$ we have
\begin{equation}\label{eq27}
\rho(B)=\rho(B)x_p=t_px_p+\sum\limits_{k=1,k\neq p}^n{\frac{b_{pk}r_kx_k}{r_p}}
=t_p+\sum\limits_{k=1}^n{\frac{a_{pk}r_kx_k}{r_p}}
\leq t_p+\frac{x_q}{r_p}\sum\limits_{k=1}^n{a_{pk}r_k}
=t_p+\frac{x_qs_p}{r_p},
\end{equation}
with equality if and only if (a) holds: (a) $x_k=x_q$ for all $k$ satisfying $1\leq k\leq n$ and $ a_{pk}\neq0$.

Similarly, we have
\begin{equation}\label{eq28}
\rho(B)x_q=t_qx_q+\sum\limits_{k=1,k\neq q}^n{\frac{b_{qk}r_kx_k}{r_q}}
=t_qx_q+\sum\limits_{k=1}^n{\frac{a_{qk}r_kx_k}{r_q}}
\leq t_qx_q+\frac{1}{r_q}\sum\limits_{k=1}^n{a_{qk}r_k}
=t_qx_q+\frac{s_q}{r_q},
\end{equation}
with equality if and only if (b) holds: (b) $x_k=x_p=1$ for all $k$ satisfying $1\leq k\leq n$ and $a_{qk}\neq0$.

Since $A$ is nonnegative irreducible, then for any $1\leq i\leq n$, there exists some $j (1\leq j\leq n)$
such that $a_{ij}>0$ and thus $r_i>0$.
Therefore, by (\ref{eq27}) and (\ref{eq28}), we have  $\rho(B)-t_p> 0$, $\rho(B)-t_q>0$ and
\begin{equation*}
\begin{aligned}
(\rho(B)-t_p)(\rho(B)-t_q)\leq 
\frac{s_ps_q}{r_pr_q}.
\end{aligned}
\end{equation*}
Then
$\rho(B)^2-(t_p+t_q)\rho(B)+t_pt_q-\frac{s_ps_q}{r_pr_q}\leq 0$,  thus
\begin{equation}\label{eq29}
\rho(B)\leq\frac{t_p+t_q+\sqrt{(t_p-t_q)^2+\frac{4s_ps_q}{r_pr_q}}}{2},
\end{equation}
and  by $a_{pq}\not=0$ we have
\begin{equation}\label{eq210}
\rho(B)\leq \max\limits_{1\leq i,j\leq n}\bigg\{\frac{t_i+t_j+\sqrt{(t_i-t_j)^2+\frac{4s_is_j}{r_ir_j}}}{2}, a_{ij}\not=0\bigg\}.
\end{equation}

Lower bounds: Without loss of generality, we can assume that one entry of $X$, say $x_p$, is equal to 1 and the others are greater than or equal to 1, i.e. $x_p=1$ and $x_k\geq 1$ for all others $k\in\{1,2,\ldots,n\}$. Let $x_q=\min\{x_k\mid a_{pk}\neq 0, 1\leq k\leq n\}$,
it is clear that $q\not= p$, $a_{pq}\not=0$ and $x_q\geq x_p$. By $R^{-1}BRX=\rho(B)X,$ we have
\begin{equation}\label{eq211}
\rho(B)=\rho(B)x_p=t_px_p+\sum\limits_{k=1,k\neq p}^n{\frac{b_{pk}r_kx_k}{r_p}}=t_p+\sum\limits_{k=1}^n{\frac{a_{pk}r_kx_k}{r_p}}
\geq t_p+\frac{x_q}{r_p}\sum\limits_{k=1}^n{a_{pk}r_k}=t_p+\frac{x_qs_p}{r_p}, 
\end{equation}
with equality if and only if $x_k=x_q$ for all $k$ satisfying $1\leq k\leq n$ and $ a_{pk}\neq 0$, and

\begin{equation}\label{eq212}
\rho(B)x_q=t_qx_q+\sum\limits_{k=1,k\neq q}^n{\frac{b_{qk}r_kx_k}{r_q}}=t_qx_q+\sum\limits_{k=1}^n{\frac{a_{qk}r_kx_k}{r_q}}
\geq t_qx_q+\frac{1}{r_q}\sum\limits_{k=1}^n{a_{qk}r_k}=t_qx_q+\frac{s_q}{r_q},
\end{equation}
with equality if and only if $x_k=x_p=1$ for all $k$ satisfying $1\leq k\leq n$ and  $a_{qk}\neq0$.

Similar to the proof of the upper bound, by (\ref{eq211}) and (\ref{eq212}), we have $\rho(B)-t_p> 0$, $\rho(B)-t_q> 0$, and
\begin{equation*}
\begin{aligned}
(\rho(B)-t_p)(\rho(B)-t_q)\geq 
\frac{s_ps_q}{r_pr_q}.
\end{aligned}
\end{equation*}
Then
$\rho(B)^2-(t_p+t_q)\rho(B)+t_pt_q-\frac{s_ps_q}{r_pr_q}\geq 0,$
  thus
\begin{equation}\label{eq213}
\rho(B)\geq \frac{t_p+t_q+\sqrt{(t_p-t_q)^2+\frac{4s_ps_q}{r_pr_q}}}{2}.
\end{equation}
and by $a_{pq}\not=0$ we have
\begin{equation}\label{eq214}
\rho(B)\geq \min\limits_{1\leq i,j\leq n}\bigg\{\frac{t_i+t_j+\sqrt{(t_i-t_j)^2+\frac{4s_is_j}{r_ir_j}}}{2}, a_{ij}\not=0\bigg\}.
\end{equation}

By (\ref{eq210}) and (\ref{eq214}), we complete the proof of (\ref{eq25}).

Now we show  the right equality in (\ref{eq25}) holds
if and only if (i) or (ii) holds.
The proof of the left equality in (\ref{eq25}) is similar, we omit it.

Sufficiency:

{\bf Case 1: } Condition (i) holds.

Since  $t_i+\frac{s_i}{r_i}=t_j+\frac{s_j}{r_j}\mbox{ for any }i\in\{1,2,\ldots,n\}$,
then $t_i-t_j=\frac{s_j}{r_j}-\frac{s_i}{r_i}$, and

$$f(i,j)=\frac{t_i+t_j+\sqrt{(t_i-t_j)^2+\frac{4s_is_j}{r_ir_j}}}{2}=t_i+\frac{s_i}{r_i},$$
thus
$ \max\limits_{1\leq i,j\leq n}\{f(i,j), a_{ij}\not=0\}=t_i+\frac{s_i}{r_i}.$

On the other hand, $R^{-1}BR$ have the same row sum $t_i+\frac{s_i}{r_i}$ for any $1\leq i\leq n$,
then we have $\rho(B)=\rho(R^{-1}BR)=t_i+\frac{s_i}{r_i }$ for any $i\in\{1,2,\ldots,n\}$  by Lemma \ref{lem23}.

Combining the above arguments, $\rho(B)=\max\limits_{1\leq i,j\leq n}\{f(i,j), a_{ij}\not=0\}=t_i+\frac{s_i}{r_i}.$

{\bf Case 2: } Condition (ii) holds.

 There exists an integer $k$ with $1\leq k<n$  such that $B$ is a partitioned matrix as (\ref{eq26}) implies that if $a_{ij}\not=0$, then $i\in \{1,\ldots, k\}$, $j\in\{k+1,\ldots, n\}$
 or $i\in\{k+1,\ldots, n\}$, $j\in \{1,\ldots, k\}$.
Take $m=t_1+\frac{ls_1}{r_1}=\ldots=t_{k}+\frac{ls_{k}}{r_{k}}=t_{k+1}+\frac{s_{k+1}}{lr_{k+1}}=\ldots=t_{n}+\frac{s_n}{lr_n}$, then

$B(r_1,\ldots,r_{k},lr_{k+1},\ldots,lr_n)^T$

\noindent\hskip.1cm $=(t_1r_1+ls_1,\ldots,t_{k}r_{k}+ls_{k},lr_{k+1}t_{k+1}+s_{k+1},\ldots,lr_nt_n+s_n)^T$

\noindent\hskip.1cm $=m(r_1,\ldots,r_{k},lr_{k+1},\ldots,lr_n)^T$.

It implies that $m$ is an eigenvalue of $B$, so $m\leq\rho(B)$.

On the other hand, it is obvious that if $a_{ij}\not=0$, then
$f(i,j)=\frac{t_i+t_j+\sqrt{(t_i-t_j)^2+\frac{4s_is_j}{r_ir_j}}}{2}=m$ 
for  any $i\in \{1,\ldots,k\}$, $j\in\{k+1,\ldots,n\}$  by $t_i-t_j=\frac{s_j}{lr_j}-\frac{ls_i}{r_i}$
or $i\in\{k+1,\ldots, n\}$, $j\in \{1,\ldots, k\}$ by $t_i-t_j=\frac{ls_j}{r_j}-\frac{s_i}{lr_i}$.
Then  we have $\rho(B)\leq \max\limits_{1\leq i,j\leq n}\{f(i,j), a_{ij}\not=0\}=m$.

Combining the above two arguments, we have
$\rho(B)=m=\max\limits_{1\leq i,j\leq n}\{f(i,j), a_{ij}\not=0\}$.

Based on the above two cases, we complete the proof of the sufficiency.

Necessity: If $\rho(B)=\max\limits_{1\leq i,j\leq n}\{f(i,j), a_{ij}\not=0\},$
then $\rho(B)\geq f(p,q)$ by $a_{pq}\not=0$, it implies
$\rho(B)=\max\limits_{1\leq i,j\leq n}\{f(i,j), a_{ij}\not=0\}=f(p,q)$ by (\ref{eq29}),
then  the equalities in (\ref{eq27}) and (\ref{eq28}) hold, and thus (a) and (b) hold.
Noting that $x_q\leq x_p=1$, we complete the proof of  necessity by the following two cases.

{\bf Case 1: }  $x_q=1$.

In this case, we will show (i) holds,
 say, we will show that $t_i+\frac{s_i}{r_i}=t_j+\frac{s_j}{r_j}$ for any $i=\{1,2,\ldots,n\}$.

Let $I^{\prime}=\{k\mid x_k=1, 1\leq k\leq n\}, I=\{1,2,\ldots,n\}$. It is clear $q, p \in I^{\prime}\subseteq I$, then $|I^{\prime}|\geq 2$.
Now we show $I^{\prime}=I$.

Otherwise,
if $I^{\prime}\neq I$, there exist  $l_1,l_2\in I^{\prime}, l_3\notin I^{\prime}$
such that $a_{l_1l_2}\neq 0$ and $a_{l_2l_3}\neq 0$ since $A$ is a nonnegative  irreducible matrix.
Therefore by $ x_{l_1}=1$ and $R^{-1}BRX=\rho(B)X,$ we have
\begin{equation}\label{eq215}
\rho(B)=\rho(B)x_{l_1}=t_{l_1}x_{l_1}+\sum\limits_{k=1, k\not=l_1}^n{\frac{b_{l_1k}x_kr_k}{r_{l_1}}}
=t_{l_1}+\frac{\sum\limits_{k=1}^n a_{l_1k}x_kr_k}{{r_{l_1}}}\leq t_{l_1}+\frac{s_{l_1}}{{r_{l_1}}}.
\end{equation}
Similarly, by $ x_{l_2}=1$, $a_{l_2l_3}\neq 0$ and $0< x_{l_3}<1$, we have
\begin{equation}\label{eq216}
\rho(B)=\rho(B)x_{l_2}
=t_{l_2}+\frac{\sum\limits_{k=1}^n a_{l_2k}x_kr_k}{{r_{l_2}}}
=t_{l_2}+\frac{\sum\limits_{k\not=l_3} a_{l_2k}x_kr_k}{{r_{l_2}}}+\frac{ a_{l_2l_3}x_{l_3}r_{l_3}}{{r_{l_2}}}
<t_{l_2}+ \frac{s_{l_2}}{r_{l_2}}.
\end{equation}

From (\ref{eq215}) and (\ref{eq216}), we have $\rho(B)-t_{l_1}>0$, $\rho(B)-t_{l_2}>0$ and
$(\rho(B)-t_{l_1})(\rho(B)-t_{l_2})<\frac{s_{l_1}s_{l_2}}{r_{l_1}r_{l_2}},$
then
$\rho(B)<f(l_1,l_2)=\frac{t_{l_1}+t_{l_2}+\sqrt{(t_{l_1}-t_{l_2})^2+\frac{4s_{l_1}s_{l_2}}{r_{l_1}r_{l_2}}}}{2},$
it implies a contradiction by the fact  $a_{l_1l_2}\neq 0$ and
$\rho(B)=\max\limits_{1\leq i,j\leq n}\{f(i,j), a_{ij}\not=0\}\geq f(l_1,l_2)$.
Thus $I^{\prime}=I$, and then  $X=(1,1,\ldots,1)^T$. Therefore,

$R^{-1}BR(1,1,\ldots,1)^T=\rho(B)(1,1,\ldots,1)^T$

\noindent $\Leftrightarrow B(R(1,1,\ldots,1)^T)=\rho(B)(R(1,1,\ldots,1)^T)$

\noindent $\Leftrightarrow B(r_1,r_2,\ldots,r_n)^T=\rho(B)(r_1,r_2,\ldots,r_n)^T$

\noindent $\Leftrightarrow t_ir_i+\sum\limits_{j=1}^na_{ij}r_j=\rho(B)r_i,  \mbox{ for any } i\in\{1,2,\ldots,n\}$

\noindent $\Leftrightarrow t_ir_i+s_i=\rho(B)r_i,  \mbox{ for any } i\in\{1,2,\ldots,n\}$

\noindent $\Leftrightarrow \frac{t_ir_i+s_i}{r_i }=\rho(B), \mbox{ for any } i\in\{1,2,\ldots,n\}$

\noindent $\Rightarrow t_i+\frac{s_i}{r_i}=t_j+\frac{s_j}{r_j}, \mbox{ for any }i\in\{1,2,\ldots,n\}.$

Based on the above arguments, (i) holds.

{\bf Case 2:  } $x_q<1$.

In this case, we will show (ii) holds,
 say, we will show that there exists an integer $k$ with $1\leq k<n$ such that $B$ is a partitioned matrix as (\ref{eq26}) and
there exists $l (0<l<1)$ such that $m=t_1+\frac{ls_1}{r_1}=\ldots=t_{k}+\frac{ls_{k}}{r_{k}}=t_{k+1}+\frac{s_{k+1}}{lr_{k+1}}=\ldots=t_{n}+\frac{s_n}{lr_n}$.

Let $N(q)=\{k\mid a_{qk}\neq 0, 1\leq k\leq n\}$, $N(p)=\{k\mid a_{pk}\neq 0, 1\leq k\leq n\}$,
 $U=\{k\mid x_k=1, 1\leq k\leq n\}$ and $W=\{k\mid x_k=x_q, 1\leq k\leq n\}$.
 So $N(q)\subseteq U$ and $N(p)\subseteq W$ by (a) and (b) hold.
 Next we will show $N(N(p))\subseteq U$ and $N(N(q))\subseteq W$.
It is obvious that $N(N(p))\not=\phi$ and $N(N(q))\not=\phi$ by $A$ thus $B$ is a nonnegative irreducible matrix.

For any $h \in N(N(p))$, there exists $h_1 \in N(p)$ such that $a_{ph_1}\neq 0$ and $a_{h_1h}\neq 0$,
where $x_{h_1}=x_q$ by $h_1 \in N(p)\subseteq W$.
By  $R^{-1}BRX=\rho(B)X,$ we have
\begin{equation}\label{eq217}
\rho(B)x_{h_1}=t_{h_1}x_{h_1}+\sum\limits_{k=1}^n{\frac{a_{h_1k}x_kr_k}{r_{h_1}}}
\leq t_{h_1}x_{h_1}+\frac{s_{h_1}}{r_{h_1}},
\end{equation}
then by (\ref{eq27}) and (\ref{eq217}), we  have
$(\rho(B)-t_{h_1})(\rho(B)-t_p)\leq \frac{s_{h_1}s_p}{r_{h_1}r_p},$
and
\begin{equation*}
\rho(B)\leq f(p, h_1)=\frac{t_{h_1}+t_p+\sqrt{(t_{h_1}-t_p)^2+\frac{4s_{h_1}s_p}{r_{h_1}r_p}}}{2}.
\end{equation*}

It implies that
$\rho(B)=f(p, h_1)$
by the fact that  $a_{ph_1}\neq 0$ and
$\rho(B)= \max\limits_{1\leq i,j\leq n}\{f(i,j), a_{ij}\not=0\}\geq f(p,h_1),$
then the equality in (\ref{eq217}) holds,  and thus $x_{h}=1$ by  $a_{h_1h}\neq 0$. Therefore we have $h\in U$ and thus $N(N(p))\subseteq U$.

Now we prove $N(N(q))\subseteq W$.
For any $h \in N(N(q))$, there exists $h_1 \in N(q)$ such that $a_{qh_1}\neq 0$ and $a_{h_1h}\neq 0$,
 where $x_{h_1}=1$ by  $h_1 \in N(q)\subseteq U$. Now we show $x_h=x_q$.

 Let $x_{q_1}=\max\{x_k \mid a_{h_1k}\neq0, 1\leq k\leq n\}$.
By  $R^{-1}BRX=\rho(B)X,$ we have

\begin{equation}\label{eq218}
\rho(B)=\rho(B)x_{h_1}=t_{h_1}x_{h_1}+\sum\limits_{k=1}^n{\frac{a_{h_1k}x_kr_k}{r_{h_1}}}
\leq t_{h_1}+x_{q_1}\frac{s_{h_1}}{r_{h_1}},\end{equation}

\begin{equation}\label{eq219}
\rho(B)x_{q_1}=t_{q_1}x_{q_1}+\sum\limits_{k=1}^n{\frac{a_{q_1k}x_kr_k}{r_{q_1}}}
\leq t_{q_1}x_{q_1}+\frac{s_{q_1}}{r_{q_1}}.\end{equation}

By (\ref{eq28}) and (\ref{eq218}), we  have
$(\rho(B)-t_{h_1})(\rho(B)-t_q)\leq \frac{x_{q_1}s_{h_1}s_q}{x_qr_{h_1}r_q}.$
Then
\begin{equation*}
\rho(B)\leq \frac{t_{h_1}+t_q+\sqrt{(t_{h_1}-t_q)^2+\frac{4x_{q_1}s_{h_1}s_q}{x_qr_{h_1}r_q}}}{2}.
\end{equation*}
It implies $x_{q_1}\geq x_q$ by the fact that $a_{qh_1}\neq 0$ and
$\rho(B)= \max\limits_{1\leq i,j\leq n}\{f(i,j), a_{ij}\not=0\}\geq f(q,h_1).$

Noting that $a_{h_1q_1}\not=0$,
by (\ref{eq218}) and (\ref{eq219}), we  have
$(\rho(B)-t_{h_1})(\rho(B)-t_{q_1})\leq\frac{s_{h_1}s_{q_1}}{r_{h_1}r_{q_1}},$
then
$\rho(B)\leq f(h_1, q_1)=\frac{t_{h_1}+t_{q_1}+\sqrt{(t_{h_1}-t_{q_1})^2+\frac{4s_{h_1}s_{q_1}}{r_{h_1}r_{q_1}}}}{2}.$
It is implies that $\rho(B)= \max\limits_{1\leq i,j\leq n}\{f(i,j), a_{ij}\not=0\}=f(h_1, q_1)$,
and thus the equalities in (\ref{eq218}) and (\ref{eq219}) hold,
it means $x_{h}=x_{q_1}\geq x_q$ for any $h \in N(N(q))$ and $x_{h_2}=1$ for any $h_2 \in N(N(N(q)))$.

Continuing the above procedure, since $B$ is a nonnegative irreducible matrix,
there exists an even number $2j$ such that $a_{q_jp}\neq0$ and $x_{q_j}\geq x_{q_{j-1}}\geq\ldots\geq x_{q_1}\geq x_q$
for any $q_j \in \underbrace{N(N\cdots(N}_{2j}(q))\cdots)$, then

\begin{equation}\label{eq220}
\rho(B)=\rho(B)x_{q_j}=t_{q_j}x_{q_j}+\sum\limits_{k=1}^n{\frac{a_{q_jk}x_kr_k}{r_{q_j}}}
\leq t_{q_j}x_{q_j}+\frac{s_{q_j}}{r_{q_j}}.\end{equation}

By (\ref{eq27}) and (\ref{eq220}), we  have
$(\rho(B)-t_{q_j})(\rho(B)-t_{p})\leq\frac{x_qs_{h_1}s_{h}}{x_{q_j}r_{q_j}r_{p}},$
so
\begin{equation*}
\rho(B)\leq \frac{t_{q_j}+t_p+\sqrt{(t_{q_j}-t_p)^2+\frac{4x_{q}s_{q_j}s_p}{x_{q_j}r_{q_j}r_p}}}{2},
\end{equation*}
it implies $x_q\geq x_{q_j}$ by the fact that $a_{q_jp}\neq0$ and
$\rho(B)= \max\limits_{1\leq i,j\leq n}\{f(i,j), a_{ij}\not=0\}\geq f(q_j,p).$
Then $x_{q_j}=\ldots=x_{q_1}=x_h=x_q$, and thus $N(N(q))\subseteq W$.

Continuing the above procedure, since $B$ is a nonnegative irreducible matrix,
it easy to see $I=U \cup W$ with  $|U|=k$ and  $|W|=n-k$,  where $I=\{1,2,\ldots,n\}$ and $1\leq k< n$.
Take $l=x_q$, then $0< l< 1$. We can assume that  $X=(\underbrace{1,\ldots,1}_{k},\underbrace{l,\ldots,l}_{n-k})^T$.

By the definitions of $p, q, N(p), N(q), N(N(p)), N(N(q))$ and $A$ thus $B$ is a nonnegative irreducible matrix,
we know both $A$ and $B$ are partitioned  matrices as (\ref{eq26}). By (\ref{eq27}) and (\ref{eq28}),
we have
$\rho(B)=t_i+\frac{ls_i}{r_i}=t_j+\frac{s_j}{lr_j}$ for any $i\in U$ and $j\in W$.

Based on the above arguments, (ii) holds.
\end{proof}

\begin{cor}\label{cor26}
Let $A=(a_{ij})$ be an $n\times n$  nonnegative  irreducible matrix with $a_{ii}=0$ for  $i=1,2,\ldots,n$, and the row sum $r_1,r_2,\ldots,r_n$.
Let $B=A+M$, where $M=diag(r_1,r_2,\ldots,r_n)$, $s_i=\sum\limits_{j=1}^n{a_{ij}r_j}$,
$\rho(B)$ be the spectral radius of $B$. Let $F(i,j)=\frac{r_i+r_j+\sqrt{(r_i-r_j)^2+\frac{4s_is_j}{r_ir_j}}}{2}$ for any $i,j\in\{1,2,\ldots,n\}$. Then
\begin{equation}\label{eq221}
\min_{1\leq i,j\leq n}\{F(i,j), a_{ij}\not=0\}\leq
\rho(B)\leq \max\limits_{1\leq i,j\leq n}\{F(i,j), a_{ij}\not=0\}.\\
\end{equation}
Moreover, one of the equalities  in (\ref{eq221}) holds if and only if one of the two conditions holds:

{\rm (i) } $r_i+\frac{s_i}{r_i}=r_j+\frac{s_j}{r_j}$ for any $i=\{1,2,\ldots,n\}$;

{\rm (ii) } There exists an integer $k$ with $1\leq k<n$ such that $B$ is a partitioned matrix as (\ref{eq26})
and there exists $l>0$ such that $r_1+\frac{ls_1}{r_1}=\ldots=r_{k}+\frac{ls_{k}}{r_{k}}=r_{k+1}+\frac{s_{k+1}}{lr_{k+1}}=\ldots=r_{n}+\frac{s_n}{lr_n}$.
In fact,  $l>1$ when the left equality holds and $l<1$ when the right equality holds.
\end{cor}

Noting that the result of the right inequality in (\ref{eq221}) was studied in \cite{2014}, and the result is the following proposition.
\begin{prop}(\cite{2014}, Theorem 2.4.)\label{prop27}
Let $A=(a_{ij})$ be an $n\times n$  nonnegative  irreducible matrix with $a_{ii}=0$ for  $i=1,2,\ldots,n$, and the row sum $r_1,r_2,\ldots,r_n$.
Let $B=A+M$, where $M=diag(r_1,r_2,\ldots,r_n)$, $s_i=\sum\limits_{j=1}^n{a_{ij}r_j}$,
$\rho(B)$ be the spectral radius of $B$. Then
\begin{equation}\label{eq222}
\rho(B)\leq \max\limits_{1\leq i,j\leq n}\bigg\{\frac{r_i+r_j+\sqrt{(r_i-r_j)^2+\frac{4s_is_j}{r_ir_j}}}{2}\bigg\}.
\end{equation}
Moreover, the equality  in (\ref{eq222}) hold if and only if  $r_i+\frac{s_i}{r_i}=r_j+\frac{s_j}{r_j}$ for any $i=\{1,2,\ldots,n\}$.
\end{prop}

Comparing the results of Corollary \ref{cor26} and Proposition \ref{prop27},
we can see that there exists some mistakes on the necessity and sufficiency conditions of the equality holds in Proposition \ref{prop27}.
The reason is that in the proof of Theorem 2.4 in \cite{2014} the necessity and sufficiency conditions of the equality
of (2.2) (and (2.3)) are incorrect, missing the condition $a_{pk}\not=0$ ($a_{qk}\not=0$).

%
%

\section{Various spectral radii of a graph }
\hskip0.6cm  Let $G$ be a connected graph,
 the (adjacency)  matrix $A(G)$,
  the signless Laplacian matrix $Q(G)$, the distance matrix $\mathcal{D}(G)$, the distance signless Laplacian matrix $\mathcal{Q}(G)$,
  the (adjacency)  spectral radius $\rho(G)$,
  the signless Laplacian spectral radius $q(G)$, the distance spectral radius $\rho^{\mathcal{D}}(G)$,
  and the distance signless Laplacian spectral radius  $q^{\mathcal{D}}(G)$
 are defined as Section 1.
In this section, we will apply Theorems \ref{T2} to $A(G)$,  $Q(G)$, $\mathcal{D}(G)$ and $\mathcal{Q}(G)$,
to obtain some new results or known results on the spectral radius.

\subsection
{Adjacency spectral radius of a graph}

\begin{theo}\label{theo35}
Let $G=(V,E)$ be a simple connected graph on $n$ vertices. Then
\begin{equation}\label{eq31}
\min\limits_{1\leq i,j\leq n}\big\{\sqrt{m_im_j}, i\sim j\big\}\leq\rho(G)\leq \max\limits_{1\leq i,j\leq n}\big\{\sqrt{m_im_j}, i\sim j\big\}.
\end{equation}
Moreover,
 one of the  equalities  in (\ref{eq31}) holds if and only if  one of the following two conditions holds:
(i) $m_1=m_2=\ldots=m_n$; (ii) $G$ is a bipartite graph and the vertices of same partition have the same  average degree.
\end{theo}

\begin{proof}
We apply  Theorem \ref{T2} to $A(G)$.

Since $t_i=0$, $a_{ii}=0$, and for $i\not=j$, $a_{ij}=\left\{\begin{array}{cc}
                                                                     1, & \mbox{ if } v_i \mbox{ and } v_j \mbox{ are adjacent;} \\
                                                                     0, & \mbox{ otherwise, }
                                                                   \end{array}\right.$
 $r_i=d_i$ and $s_i=\sum\limits_{i\sim k}{d_k}=d_im_i$ for  any $1\leq i\leq n$,
 then $\sqrt{\frac{s_is_j}{r_ir_j}}=\sqrt{\frac{\sum\limits_{i\sim k}{d_k}\sum\limits_{j\sim k}{d_k}}{d_id_j}}=\sqrt{m_im_j}$,
thus (\ref{eq31}) holds by (\ref{eq25}).

Furthermore, $t_i+\frac{s_i}{r_i}=t_j+\frac{s_j}{r_j}$ for all $i,j\in\{1,2,\ldots, n\}$ implies $m_1=m_2=\ldots=m_n$.
 $B=A(G)$ is a partitioned matrix implies that $G$ is a bipartite graph,
and $t_1+\frac{ls_1}{r_1}=\ldots=t_{k}+\frac{ls_{k}}{r_{k}}=t_{k+1}+\frac{s_{k+1}}{lr_{k+1}}=\ldots=t_{n}+\frac{s_n}{lr_n}$
implies that the vertices of same partition have the same  average degree.
Thus one of the  equalities  in (\ref{eq31}) holds if and only if (i) or (ii) holds.
\end{proof}

\begin{rem}
The right inequality in Theorem \ref{theo35} is the result of Theorem 2.3 in \cite{2004Das2}.
\end{rem}

\subsection{Signless Laplacian  spectral radius of a graph }

\begin{lem}(\cite{2004Das}, Lemma 2.3.)\label{lem310}
Let $G=(V,E)$ be a simple connected graph with  vertex set $V=\{v_1,v_2,\ldots,v_n\}$. For any $v_i\in V$, the degree of $v_i$ and the average degree of the vertices adjacent to $v_i$ are denoted by $d_i$ and $m_i$, respectively. Then $d_1+m_1=d_2+m_2=\ldots =d_n+m_n$ holds if and only if $G$ is a regular graph or a bipartite semi-regular graph.
\end{lem}

\begin{theo}\label{thm311}
 Let $G=(V,E)$ be a connected graph on $n$ vertices, for any $1\leq i,j\leq n$, $g(i,j)=\frac{d_i+d_j+\sqrt{(d_i-d_j)^2+4m_im_j}}{2}$. Then

\begin{equation}\label{eq32}
\min\{g(i,j), i \sim j\} \leq q(G)\leq \max\{g(i,j), i \sim j\},
\end{equation}
and one of the equalities in (\ref{eq32})  holds if and only if one of the following conditions holds: (1) $G$ is a regular graph;
 (2) $G$ is a bipartite semi-regular graph; (3) $G$ is a bipartite graph and there exists an integer $k$ with $1\leq k<n$
and a real number  $l>0$ such that $d_1+lm_1=\ldots=d_{k}+lm_k=d_{k+1}+\frac{m_{k+1}}{l}=\ldots=d_{n}+\frac{m_n}{l}$.
In fact,  $l>1$ when the left equality holds and $l<1$ when the right equality holds.

\end{theo}
\begin{proof}
 We apply Theorem \ref{T2} to $Q(G)$. 
 
 Since $t_i=r_i=d_i$, $a_{ii}=0$, and for $i\not=j$, $a_{ij}=\left\{\begin{array}{cc}
                                                                     1, & \mbox{ if } v_i \mbox{ and } v_j \mbox{ are adjacent;} \\
                                                                     0, & \mbox{ otherwise, }
                                                                   \end{array}\right.$
 $s_i=\sum\limits_{i\sim k}{d_k}=d_im_i$ for  $i,j\in\{1,2,\ldots,n\}$,
 then $$\frac{t_i+t_j+\sqrt{(t_i-t_j)^2+\frac{4s_is_j}{r_ir_j}}}{2}=\frac{d_i+d_j+\sqrt{(d_i-d_j)^2+4m_im_j}}{2},$$
thus (\ref{eq32}) holds
by (\ref{eq25}).

Furthermore, by Theorem \ref{T2} we know one of the equalities in (\ref{eq32}) holds  if and only if one of the two conditions hold: (I) $d_i+m_i=d_j+m_j$ for all $i,j\in\{1,2,\ldots, n\}$; (II) there exists an integer $k$ with $1\leq k<n$
such that $Q(G)$ is a partitioned matrix as (\ref{eq26}) and there exists $l>0$ such that $d_1+lm_1=\ldots=d_{k}+lm_k=d_{k+1}+\frac{m_{k+1}}{l}=\ldots=d_{n}+\frac{m_n}{l}$,
where $l>1$ when the left equality holds and $l<1$ when the right equality holds.

Noting that $d_i+m_i=d_j+m_j$ for all $i,j\in\{1,2,\ldots, n\}$ 
 if and only if $G$ is a regular or bipartite semi-regular graph by Lemma \ref{lem310},
  and  $Q(G)$ is a partitioned matrix as (\ref{eq26}) if and only if  $G$ is a bipartite graph,
  so we complete the proof. \end{proof}

\begin{prop}\label{prop312}(\cite{2013Maden}, Theorem 6.)
Let $G=(V,E)$ be a connected graph on $n$ vertices, $g(i,j)=\frac{d_i+d_j+\sqrt{(d_i-d_j)^2+4m_im_j}}{2}$ for any $1\leq i,j\leq n$.
Then (\ref{eq32}) holds,
 and the equality if and only if $G$ is a regular graph or a bipartite semi-regular graph.
\end{prop}

Comparing the results of Theorem \ref{thm311} and Proposition \ref{prop312},
we can see that there are different on the conditions when the equality holds.
In fact, if $G$ is a bipartite semi-regular graph, we can see condition (3) of  Theorem \ref{thm311} holds.
But when condition (3) of  Theorem \ref{thm311} holds, we do not decide whether $G$ is a bipartite semi-regular graph or not.
Even we try to find an example to say ``yes" or ``no", but we failed.
Thus it is natural to propose the following question.

\begin{prob}
 Let $G=(V,E)$ be a connected bipartite graph. Then $G$ is a semi-regular graph if and only if there exists  an integer $k$ with $1\leq k<n$
and a real number  $l>0$  such that $d_1+lm_1=\ldots=d_{k}+lm_k=d_{k+1}+\frac{m_{k+1}}{l}=\ldots=d_{n}+\frac{m_n}{l}$?
\end{prob}

\subsection{Distance spectral radius of  a graph}

\begin{theo}\label{theo315}
Let $G=(V,E)$ be a connected graph on $n$ vertices,
$T_1,T_2,\ldots,T_n$ be the second distance degree sequence of $G$. Then
\begin{equation}\label{eq34}
\min\limits_{1\leq i,j\leq n}\bigg\{\sqrt{\frac{T_iT_j}{D_iD_j}}\bigg\}\leq\rho^D(G)\leq \max\limits_{1\leq i,j\leq n}\bigg\{\sqrt{\frac{T_iT_j}{D_iD_j}}\bigg\},
\end{equation}
and one of the equality in (\ref{eq34}) holds if and only if  $\frac{T_1}{D_1}=\frac{T_2}{D_2}=\ldots=\frac{T_n}{D_n}$.
\end{theo}

\begin{proof}
We apply Theorem \ref{T2} to $\mathcal{D}(G)$. Since $t_i=0$, $a_{ij}=d_{ij}\not=0$ for all $i\not=j$,
$a_{ii}=d_{ii}=0$,
 $r_i=\sum\limits_{j=1}^n d_{ij}=D_i$ and $s_i=\sum\limits_{j=1}^n{d_{ij}D_j}=T_i$ for  $i=1,2,\ldots,n$, then $\sqrt{\frac{s_is_j}{r_ir_j}}=\sqrt{\frac{T_iT_j}{D_iD_j}}$,
and thus (\ref{eq34}) holds by (\ref{eq25}).

Since $a_{ij}=d_{ij}\not=0$ for all $i\not=j$, then $\mathcal{D}(G)$ is not a partitioned matrix as (\ref{eq26}),
thus the equality holds if and only if $t_i+\frac{s_i}{r_i}=t_j+\frac{s_j}{r_j}$ for all $i,j\in\{1,2,\ldots, n\}$,
say $\frac{T_1}{D_1}=\frac{T_2}{D_2}=\ldots=\frac{T_n}{D_n}$ for all $i,j\in\{1,2,\ldots, n\}$.
\end{proof}
\begin{rem}
The right inequality in Theorem \ref{theo315} is the result of Theorem 2.3 in \cite{2010He}.
\end{rem}

\subsection{Distance signless Laplacian spectral radius of  a graph}

\vskip.1cm
\begin{theo}\label{thm318}
 Let $G=(V,E)$ be a connected graph on $n$ vertices, for all $1\leq i,j\leq n$, $h(i,j)=\frac{D_i+D_j+\sqrt{(D_i-D_j)^2+\frac{4T_iT_j}{D_iD_j}}}{2}$. Then
  \begin{equation}\label{eq37}
 \min_{1\leq i,j\leq n}\{h(i,j)\} \leq q^D(G)\leq \max_{1\leq i,j\leq n}\{h(i,j)\},
\end{equation}
and the equality holds if and only if  $D_1+\frac{T_1}{D_1}=D_2+\frac{T_2}{D_2}=\ldots=D_n+\frac{T_n}{D_n}$.
\end{theo}

\begin{proof}
We apply Theorem \ref{T2} to $\mathcal{Q}(G)$. Since $r_i=t_i=D_i$, $a_{ij}=d_{ij}\not=0$ for all $i\not=j$,
$a_{ii}=d_{ii}=0$, and $s_i=\sum\limits_{j=1}^{n}{d_{ij}D_j}=T_i$ for all $i=1,2,\ldots,n$,
 then $ \frac{t_i+t_j+\sqrt{(t_i-t_j)^2+\frac{4s_is_j}{r_ir_j}}}{2}=\frac{D_i+D_j+\sqrt{(D_i-D_j)^2+\frac{4T_iT_j}{D_iD_j}}}{2}$,
thus (\ref{eq37}) holds by (\ref{eq25}).

Since $a_{ij}=d_{ij}\not=0$ for all $i\not=j$, then $\mathcal{D}(G)$ is not a partitioned matrix as (\ref{eq26}),
thus the equality holds if and only if $t_i+\frac{s_i}{r_i}=t_j+\frac{s_j}{r_j}$ for all $i,j\in\{1,2,\ldots, n\}$,
say $D_1+\frac{T_1}{D_1}=D_2+\frac{T_2}{D_2}=\ldots=D_n+\frac{T_n}{D_n}$.
\end{proof}

\begin{rem}
The right inequality in Theorem \ref{thm318} is the result of Theorem 3.7 in \cite{2014}.
\end{rem}

\section{Various spectral radii of a digraph }
\hskip0.6cm  Let $\overrightarrow{G}$ be a strong connected digraph,
 the adjacency  matrix $A(\overrightarrow{G})$,
  the signless Laplacian matrix $Q(\overrightarrow{G})$, the distance matrix $\mathcal{D}(\overrightarrow{G})$, the distance signless Laplacian matrix $\mathcal{Q}(\overrightarrow{G})$,
 and the adjacency  spectral radius $\rho(\overrightarrow{G})$,
  the signless Laplacian spectral radius $q(\overrightarrow{G})$, the distance spectral radius $\rho^{\mathcal{D}}(\overrightarrow{G})$, the distance signless Laplacian spectral radius  $q^{\mathcal{D}}(\overrightarrow{G})$
 are defined as Section 1.
In this section, we will apply Theorems \ref{T2} to $A(\overrightarrow{G})$,  $Q(\overrightarrow{G})$, $\mathcal{D}(\overrightarrow{G})$
 and $\mathcal{Q}(\overrightarrow{G})$, to obtain some new results or known results on the spectral radius.

\subsection{Adjacency spectral radius of a digraph}

\vskip.1cm
\begin{theo}\label{thm42}
(\cite{2009}, Theorem 2.1 and Theorem 2.2)
Let $\overrightarrow{G}=(V,E)$ be a strong connected digraph on $n$ vertices. Then
\begin{equation}\label{eq41}
\min_{1\leq i,j\leq n}\big\{\sqrt{m_i^{+}m_j^{+}}, i\sim j\big\}\leq\rho(\overrightarrow{G})\leq \max\limits_{1\leq i,j\leq n}\{\sqrt{m_i^{+}m_j^{+}}, i\sim j\big\},
\end{equation}
and one of the equalities holds if and only if one of the following two conditions holds: (i) $m_1^+=m_2^+=\ldots=m_n^+$, (ii) $\overrightarrow{G}$ is a bipartite graph and the vertices of same partition have the same  average outdegree.
\end{theo}

\begin{proof}
We apply Theorem \ref{T2} to $A(\overrightarrow{G})$.

Since $t_i=0$, $a_{ii}=0,$ for $i\not=j$, $a_{ij}=\left\{\begin{array}{cc}
                                                                     1, & \mbox{ if } (v_i, v_j )\in E; \\
                                                                     0, & \mbox{ otherwise,}
                                                                   \end{array}\right.$
 $r_i=d^+_i$ and $s_i=\sum\limits_{i\sim k}{d^+_k}=d^+_im^+_i$ for  $i=1,2,\ldots,n$,
 then $\sqrt{\frac{s_is_j}{r_ir_j}}=\sqrt{m^+_im^+_j}$,
thus (\ref{eq41}) holds by (\ref{eq25}).

Furthermore, $t_i+\frac{s_i}{r_i}=t_j+\frac{s_j}{r_j}$ for all $i,j\in\{1,2,\ldots, n\}$ implies $m_1^+=m_2^+=\ldots=m_n^+$.
Moreover, $B=A(G)$ is a partitioned matrix implies that $\overrightarrow{G}$ is a bipartite graph,
and $t_1+\frac{ls_1}{r_1}=\ldots=t_{m}+\frac{ls_{m}}{r_{m}}=t_{m+1}+\frac{s_{m+1}}{lr_{m+1}}=\ldots=t_{n}+\frac{s_n}{lr_n}$
implies that the vertices of same partition have the same  average outdegree.
Thus one of the  equalities  in (\ref{eq41}) holds if and only if (i) or (ii) holds.
\end{proof}

\subsection{Signless Laplacian spectral radius of a digraph}

\vskip.1cm
\begin{theo}\label{thm44} (\cite{2013}, Theorem 3.2.)
Let $\overrightarrow{G}=(V,E)$ be a strong connected digraph on $n$ vertices,
$G(i,j)=\frac{d^+_i+d^+_j+\sqrt{(d^+_i-d^+_j)^2+4m^+_im^+_j}}{2}$ for any $i,j\in\{1,2,\ldots,n\}$. Then

\begin{equation}\label{eq42}
\min_{1\leq i,j\leq n}\{G(i,j),  i\sim j\}
\leq q(\overrightarrow{G})\leq \max\limits_{1\leq i,j\leq n}\{G(i,j), i\sim j\}.
\end{equation}

\end{theo}
\begin{proof}
We apply Theorem \ref{T2} to $Q(\overrightarrow{G})$. Since $t_i=d^+_i$,  for $i\not=j$, $a_{ij}=\left\{\begin{array}{cc}
                                                                     1, & \mbox{ if } (v_i, v_j )\in E; \\
                                                                     0, & \mbox{ otherwise, }
                                                                   \end{array}\right.$
$a_{ii}=0$, $r_i=d^+_i$ and $s_i=\sum\limits_{i\sim k}{d^+_k}=d^+_im^+_i$ for  $i=1,2,\ldots,n$,
 then $ \frac{t_i+t_j+\sqrt{(t_i-t_j)^2+\frac{4s_is_j}{r_ir_j}}}{2}=\frac{d^+_i+d^+_j+\sqrt{(d^+_i-d^+_j)^2+{4m^+_im^+_j}}}{2}$,
thus (\ref{eq42}) holds by (\ref{eq25}).
\end{proof}

\begin{rem}
By Theorem \ref{T2}, we conclude that one of the  equalities in Theorem \ref{thm44}  holds if and only if one of the following two conditions holds:
 (i) $d^+_1+m^+_1=d^+_2+m^+_2=\ldots=d^+_n+m^+_n$.
 (ii) there exists an integer $k$ with $1\leq k<n$ such that $A(\overrightarrow{G})$ is a partitioned matrix as (\ref{eq26}),
 and there exists a real number $l>0$ such that $d^+_1+lm^+_1=\ldots=d^+_{k}+lm^+_k=d^+_{k+1}+\frac{m^+_{k+1}}{l}=\ldots=d^+_{n}+\frac{m^+_n}{l}$.
\end{rem}

\subsection{Distance spectral radius of  a digraph }

\begin{theo}\label{thm47}
Let $\overrightarrow{G}=(V,E)$ be a strong connected digraph on $n$ vertices,
$T^+_1,T^+_2,\ldots,T^+_n$ be the second distance out-degree sequence of $\overrightarrow{G}$. Then
\begin{equation}\label{eq44}
\min\limits_{1\leq i,j\leq n}\bigg\{\sqrt{\frac{T^+_iT^+_j}{D^+_iD^+_j}}\bigg\}\leq\rho^D(\overrightarrow{G})\leq \max\limits_{1\leq i,j\leq n}\bigg\{\sqrt{\frac{T^+_iT^+_j}{D^+_iD^+_j}}\bigg\},
\end{equation}
and one of the equalities holds if and only if $\frac{T^+_1}{D^+_1}=\ldots=\frac{T^+_n}{D^+_n}$.
\end{theo}

\begin{proof}
We apply Theorem \ref{T2} to $\mathcal{D}(\overrightarrow{G})$. Since $t_i=0$, $a_{ij}=d_{ij}\not=0$ for all $i\not=j$,
$a_{ii}=d_{ii}=0$,  $r_i=D^+_i=\sum\limits_{j=1}^n d_{ij}$ and $s_i=\sum\limits_{j=1}^n{d_{ij}D^+_j}=T_i^+$ for  $i=1,2,\ldots,n$,
 then $\sqrt{\frac{s_is_j}{r_ir_j}}=\sqrt{\frac{T^+_iT^+_j}{D^+_iD^+_j}}$,
thus (\ref{eq44}) holds by (\ref{eq25}) and $a_{ij}=d_{ij}\not=0$ for all $i\not=j$.

Since $a_{ij}=d_{ij}\not=0$ for all $i\not=j$, then $\mathcal{D}(\overrightarrow{G})$ is not a partitioned matrix as (\ref{eq26}),
thus one of the equalities holds if and only if $t_i+\frac{s_i}{r_i}=t_j+\frac{s_j}{r_j}$ for all $i,j\in\{1,2,\ldots, n\}$,
say $\frac{T^+_1}{D^+_1}=\ldots=\frac{T^+_n}{D^+_n}$.
\end{proof}

\subsection{Distance signless Laplacian spectral radius of a diagraph}

\begin{theo}\label{thm49}
Let $\overrightarrow{G}=(V,E)$ be a strong connected digraph on $n$ vertices,  for any $i,j\in\{1,2,\ldots,n\}$,
$H(i,j)=\frac{D^+_i+D^+_j+\sqrt{(D^+_i-D^+_j)^2+\frac{4T_i^{+}T_j^+}{D_i^{+}D_j^{+}}}}{2}$. Then
\begin{equation}\label{eq47}
\min_{1\leq i,j\leq n}\{H(i,j)\}\leq q^D(\overrightarrow{G})\leq \max\limits_{1\leq i,j\leq n}\{H(i,j)\},
\end{equation}
and the equality holds if and only if $D^+_i+\frac{T^+_i}{D^+_i}=D^+_j+\frac{T^+_j}{D^+_j}$ for any $i,j\in\{1,2,\ldots,n\}$.
\end{theo}

\begin{proof}
We apply Theorem \ref{T2} to $\mathcal{Q}(\overrightarrow{G})$. Since $t_i=D^+_i$, $a_{ij}=d_{ij}\not=0$ for all $i\not=j$,
$a_{ii}=d_{ii}=0$,  $r_i=D^+_i=\sum\limits_{j=1}^n d_{ij}$ and $s_i=\sum\limits_{j=1}^n{d_{ij}D^+_j}=T_i^+$ for  $i=1,2,\ldots,n$,
 then $\frac{t_i+t_j+\sqrt{(t_i-t_j)^2+\frac{4s_is_j}{r_ir_j}}}{2}=\frac{D^+_i+D^+_j+\sqrt{(D^+_i-D^+_j)^2+\frac{4T^+_iT^+_j}{D^+_iD^+_j}}}{2}$,
thus (\ref{eq47}) holds by (\ref{eq25}) and $a_{ij}=d_{ij}\not=0$ for all $i\not=j$.

Since $a_{ij}=d_{ij}\not=0$ for all $i\not=j$, then $\mathcal{Q}(\overrightarrow{G})$ is not a partitioned matrix as (\ref{eq26}),
thus the equality holds if and only if $t_i+\frac{s_i}{r_i}=t_j+\frac{s_j}{r_j}$ for all $i,j\in\{1,2,\ldots, n\}$,
say $D^+_1+\frac{T^+_1}{D^+_1}=\ldots=D^+_n+\frac{T^+_n}{D^+_n}$.
\end{proof}

\end{document}